\documentclass[12pt]{scrartcl}

\usepackage{array}
\usepackage{supertabular}
\usepackage{amsmath,amsfonts,amssymb,amsthm} 
\usepackage{bm}
\usepackage{graphicx}
\usepackage{subcaption}

\newcommand{\xwai}{\sum_{i=0}^p}

\allowdisplaybreaks[4]

\newtheorem{lemma}{Lemma}
\newtheorem{theo}{Theorem}

\newtheorem{coro}{Corollary}

\title {Efficiency of maximum likelihood estimation for a multinomial distribution with known probability sums}
\author{Yo Sheena\thanks{Faculty of Data Science, Shiga University, Japan; Visiting Professor of the Institute of Statistical Mathematics, Japan. E-mail: yo-sheena@biwako.shiga-u.ac.jp. }}

\date{August, 2020}

\begin{document}
\maketitle

\begin{abstract}
\ \ For a multinomial distribution, suppose that we have prior knowledge on the sum of the probabilities of some categories. This allows us to construct a submodel in a full (i.e., no-restriction) model. Maximum likelihood estimation (MLE) under this submodel is expected to have better estimation efficiency than MLE under the full model. This article presents the asymptotic expansion of the risk of MLE with respect to Kullback--Leibler divergence for both the full model and the ``m-aggregation'' submodel. The result reveals that the second-order term ($n^{-2}$-order term) of the submodel is larger than that of the full model unless the submodel is constructed from ``solid'' prior knowledge. We also present some concrete cases that the use of the submodel can increase the risk. 
\end{abstract}
\noindent
MSC(2010) \textit{Subject Classification}: Primary 60F99; Secondary 62F12
\section{Introduction}
\label{section:int}
\ \ As Eguchi \cite{Eguchi1} , \cite{Eguchi2} showed, the geometrical properties of the manifold of a parametric distribution model can be induced by a divergence on it. Hence, if we make an analysis for a statistical model that is characterized by some certain geometrical properties (especially flatness), it seems natural that we use the corresponding divergence. 

Let's consider this issue in view of the submodel.  Consider an estimation problem in a parametric model $\mathcal{M}$ with the parameter $\theta$. If we have prior knowledge on $\theta$, it gives a submodel, $\mathcal{M^*}$. We usually project an estimator (say $\hat{\theta}$) in $\mathcal{M}$ onto $\mathcal{M^*}$, and use the projected point as an alternative estimator, $\hat{\theta}^*$. Naturally we conjecture that $\hat{\theta}^*$ performs better than $\hat{\theta}$ because of the use of the prior knowledge. This is theoretically guaranteed under some conditions; For example, if we use Bregman divergence induced from a convex function $\Psi(\theta)$ (say $D_\Psi[ \cdot \, | \, \cdot]$) and $\mathcal{M^*}$ is the affine subspace with respect to the dual parameter $\eta$, then the generalized Pythagorean theorem (see Amari \cite{Amari4}) says 
\[
D_\Psi[\hat{\theta} \, | \, \theta^* ] = D_\Psi[\hat{\theta} \, | \, \hat{\theta}^*] + D_\Psi[\hat{\theta}^* \, | \, \theta^*],
\]
where $\theta^*$ is the parameter of the true distribution, and $\hat{\theta}^*$ is the projected point with respect to $D_\Psi[\hat{\theta} \, | \, \cdot]$, namely
\begin{equation}
\label{psi_projection}
\hat{\theta}^* =\displaystyle{\operatorname*{\mathrm{argmin}}_{\theta \in \mathcal{M^*}}} D_\Psi[ \hat{\theta} \, | \, \theta].
\end{equation}
Consequently the risk of $\hat{\theta}^*$ is smaller than that $\hat{\theta}$, i.e.
\begin{equation}
\label{ineq_risk}
E[D_\Psi[\hat{\theta} \, | \, \theta^* ] ] \geq E[D_\Psi[\hat{\theta}^* \, | \, \theta^*]],
\end{equation}
where the expectation is taken under the true distribution. The point is pairing of $D_\Psi[\cdot \,|\, \cdot]$ and the submodel that is affine with respect to its dual coordinate.

This paper investigates the risk of the submodel for the case the pair is mismatched. The framework of the paper is given as follows:
\begin{enumerate}
\item The object of the estimation is the finite dimensional discrete distribution, i.e. the multinomial distribution.
\item The prior knowledge is completely correct and is given as the sum of the probabilities of some bins (categories).
\item The estimation method is MLE and we use Kullback-Leibler divergence (K-L divergence) for the estimation evaluation.
\end{enumerate}
The statistical problems under this framework are often observed in the surveys where a prior survey is available. It is an important issue whether we should adjust the result of the present survey according to the prior survey. We show two concrete examples:
\\
\\
-- Household Survey --
\\
\ \  From a household sample survey, we have formulated a two-way contingency table by categorizing each household in view of the yearly income level  and the age group of the main earner.  The relative frequency of each category to the sample size gives us MLE estimation of the corresponding two-dimensional discrete distribution. Suppose that a national census was carried out in the previous year,  then we have the highly reliable knowledge on the distribution of the age group. This means we have the prior knowledge on the sum of the probabilities for each row (or column) in the two-way table. The natural question arises whether we should modify the survey result so that the row (column) sums are equal to those of the census. 
\\
\\
--  Land Survey --
\\
\ \  We are interested in the state of land in a certain district. Land is categorized into three kinds of areas: non-forest land area, commercial forest land area, noncommercial forest land area. We will make a field investigation from randomly chosen points and make an estimation of the ratio among the above three areas. Prior to investigation, we already have obtained an aerial photo of the district, which informs us about the ratio between the forest and non-forest land areas. Hence, we have the prior information on the probability of one category (non-forest area) of the multinomial distribution. The same question as above arises.
\\

Now we formulate the estimation problem. First we consider the multinomial distribution as follows; 
In each independent trial, the random variable $X$ takes a value which belongs to one of the categories $C_i\ (i=0,1,\ldots,p)$ with the probability
\begin{equation}
\label{full_model}
m\triangleq (m_0, \ldots, m_p),\qquad m_i \triangleq P(X \in C_i)>0, \quad i=0,1,\ldots,p.
\end{equation}
Due to the restriction
\begin{equation}
\label{prob_cond}
\xwai m_i =1,
\end{equation}
the dimension of this probability model is equal to $p$. The model with no other  restrictions than \eqref{full_model} will be called "full model." 

Suppose we have prior information on the sum of the probabilities over some categories, namely
\begin{equation}
\label{info_sum}
m_{i_1} + \cdots +m_{i_s} = c, \qquad  0 \leq i_1 <  i_2 < \ldots < i_s  \leq p
\end{equation}
with some known constant $c$. If we have a collection of such restriction equations as \eqref{info_sum}, it formulates a submodel (hereafter we call it as "$m$-aggregation submodel").
If each parameter $m_i$ appears no more than once in all restriction equations, we call the model ``non-overlapping.'' 
If some $m_i$'s appear multiple times in the restriction, the model is referred to as ``overlapping.'' 
For example, if we only know the column sums of a two-way contingency table, the submodel is non-overlapping. 
If both row- and column-wise sums are known, an overlapping submodel is formulated.
Especially if we know the probability of one category, that is, $s=1$ in \eqref{info_sum}, we call the restriction (or information) is ``solid.''

Hereafter we denote the MLE of the parameter $m$ (in the full model or submodel according to the context) by
\begin{equation}
\label{MLE_m}
\widehat{m}\triangleq (\widehat{m}_0, \ldots, \widehat{m}_p),
\end{equation}
and measure the discrepancy between the true distribution and the predictive distribution using K-L divergence,
\begin{equation}
\label{K-Ldive}
D[\widehat{m}:m]=\sum_{i=0}^p \widehat{m}_i \log \frac{\widehat{m}_i}{m_i}
\end{equation}
Note that the MLE for the submodel is given by the projection onto the submodel of the MLE of the full model. (See \cite{Gokhale&Kullback} for the use of Kullback--Leibler divergence in the inference for the multinomial distribution.)

We evaluate the performance of MLE $\widehat{m}$ through the risk, that is 
\begin{equation}
\label{def_ED^-1}
ED=ED[\widehat{m}: m]\triangleq E\bigl[D[\widehat{m}: m]\bigr].
\end{equation}

Here we show that the $m$-aggregation submodel and K-L divergence are ``mismatched''.  Treat $m$ as an element of the set of $p+1$-dimensional positive measures, i.e. 
\[
\bar{\mathcal{M}} \triangleq \bigl\{m=(m_0,\ldots, m_p) \,|\, m_i > 0, \ i=0,\ldots,p \bigr\}, 
\]
where the sum is not  necessarily equal to one. Let 
\begin{equation}
\label{def_theta}
\theta^i \triangleq \log{m_i}+1,\qquad i=0, \ldots,p
\end{equation}
and 
\[
\Psi(\theta) \triangleq \sum_{i=0}^p e^{\theta^i-1}= \sum_{i=0}^p m_i.
\]
Then $\Psi(\theta)$ is the convex function of  $\theta=(\theta^0,\ldots,\theta^p)$, and as
\[
\frac{\partial \quad}{\partial \theta^i} \Psi(\theta) = e^{\theta^i-1} = m_i, \qquad i=0,\ldots,p,
\]
$m$ is the dual coordinate of $\theta$. The divergence  induced from $\Psi(\theta)$ on $\bar{\mathcal{M}}$ equals 
\begin{equation}
\label{conjugate_KL}
\begin{split}
D_{\Psi}[\theta_1\,|\, \theta_2]  &= \Psi(\theta_1) - \Psi(\theta_2) - \sum_{j=0}^p m^2_j (\theta_1^j-\theta_2^j) \\
&= \sum_{i=0}^p m_i^1 - \sum_{i=0}^p m_i^2 + \sum_{i=0}^p m^2_i \log{(m_i^2/m_i^1)}
\end{split}
\end{equation}
for $\theta_i = (\theta_i^0,\ldots,\theta_i^p),$ and their dual coordinates $m^i= (m^i_0,\ldots, m^i_p),\ i=1,2$. If the two points $\theta_1$ and $\theta_2$ are probabilities, i.e. $\sum_{j=0}^p m ^{i}_j=1$ for $i=1,2$, then
\begin{equation}
\label{conjugate_KL2}
D_{\Psi}[\theta_1\,|\, \theta_2] = \sum_{j=0}^p m_j^2 \log{(m_j^2/m_j^1)}.
\end{equation}

This is the conjugate of K-L divergence, and not K-L divergence itself. Actually
\[
D[\theta_2\,|\, \theta_1] = D_{\Psi}[\theta_1\,|\, \theta_2].
\]
Since the $m$-aggregation submodel is an affine subspace in $\bar{\mathcal{M}}$ with respect to $m$, if we use $D_{\Psi}[\cdot \,|\, \cdot]$, Pythagorean theorem guarantees the smaller risk of $\hat{\theta}_*$ than $\hat{\theta}$. However if we use K-L divergence, this is not theoretically guaranteed. 

In this paper,  the asymptotic expansion of risk is derived for non-overlapping $m$-aggregation submodel with respect to the sample size $n$ up to the second-order term (Section \ref{sec:two-stage}). The first-order term ($n^{-1}$-order term) and the second-order term ($n^{-2}$-order term) provide information on the asymptotic efficiency of the MLE for the submodel. 

Sheena \cite{Sheena} derived the asymptotic expansion of the risk of MLE with regard to $\alpha$-divergence for a general  parametric model that is well-specified (namely the model that contains the true distribution). (See Theorem 1 of \cite{Sheena}.)
As an application to the full model multinomial distribution, it demonstrated (see (42) of \cite{Sheena}) that
\begin{equation}
\label{ED_-1_expan_full}
ED=\frac{p}{2n}+\frac{1}{12n^2}(M-1)+o(n^{-2}),
\end{equation}
where 
\begin{equation}
\label{def_M}
M \triangleq \sum_{i=0}^p m^{-1}_i.
\end{equation}
The first-order term is proportional to the ratio of the model's dimension to the sample size (``$p-n$ ratio''), which indicates (relative) complexity of the model. Since this holds true for any parametric model, the first-order term for any well-specified submodel  is smaller than that of the full model, which is the benefit of the dimension reduction for the submodel.  Consequently when the sample size is large enough, the risk of the submodel is smaller than that of the full model, and as the sample size goes to infinity, the risk ratio between the sub and full models converges to the dimensions ratio.

However as the result of this paper (Corollary \ref{coro_second_order_term}) shows, the second-order term of the non-overlapping $m$-aggregation submodel is always larger than that of the full model unless the prior information is solid. This indicates if the sample size $n$ is not large enough to neglect the second-order term, the risk reduction of the submodel does not become as small as we expect from the dimension reduction, or even anomaly, reversed risk order, could happen. In Section \ref{example}, simulation studies demonstrate that these phenomena actually happen.

Improvement of the inference efficiency using prior knowledge is an important theme also in hypothesis testing. For this issue,  we refer to Trosset et. al. \cite{Trosset_et_al}. They treated Marden's conjecture that the additional information on the alternative hypothesis increase the power of the likelihood test. They showed that this conjecture fails in some concrete distributions even when the information reduces the dimension in the alternative hypothesis.
\section{Risk of non-overlapping $m$-aggregation submodel}
\label{risk_m_family}
\subsection{Two-stage multinomial distribution models}
\label{sec:two-stage}
\ \ For the risk analysis of a non-overlapping submodel, it is convenient to decompose each trial (experiment) into  two stages. In this subsection, we consider the ``two-stage multinomial distribution model'.

We begin by defining the two-stage multinomial distribution. Suppose that the random variable $X$ takes values that belong to one of the categories $C_{ij}$ and that
\begin{equation}
m_{ij}\triangleq P(X \in C_{ij})\ (>0),\qquad m_{i\cdot}\triangleq \sum_{j=1}^{J_i} m_{ij}, \qquad p_{ij}\triangleq \frac{m_{ij}}{m_{i\cdot}}
\end{equation}
for $i=1,\ldots,I, \ j=1,\ldots, J_i$. (To eliminate a trivial case, we suppose that $I\geq 2$.)
Let $C_{i\cdot} \triangleq \bigcup\limits_{j} C_{ij}$. Then, the first-stage model is given by focusing on which $C_{i\cdot}$ the value of $X$ belongs to. Its parameters are given by $m_{i\cdot}\ (i=1,\ldots, I)$. The $i$th model ($i=1,\ldots, I$) in the second stage is given by the categories $C_{ij}$ and the corresponding probabilities $p_{ij}$ ($j=1,\ldots,J_i$) under the condition $X \in C_{i\cdot}$. 

%If we have some prior knowledge on the first-stage model such as $m_{i\cdot}=b_i, (1\leq i \leq t,\ 1\leq \exists t \leq I)$, the model as a whole is a $m$-aggregation submodel.
If the restriction for the submodel is non-overlapping, the parameters $m_{ij}$ can be grouped by the restriction equation in which the parameter appears (the parameters that do not appear in any restrictions form one group). The full model is decomposed into two stages according to this grouping, and the full model can be considered as a two-stage model. 

Kullback--Leibler divergence satisfies the so-called ``chain rule.'' Let $(X | U)$ denote the conditional distribution of $X$ when $U$ is given. Suppose a pair of the distributions $(X | U)_i$ and $U_i$ defines the distribution of $X_i$ for $i=1,2$. Then, the following relationship (chain rule) holds.
\begin{equation}
\label{chain_rule}
D[X_1 : X_2] = D[U_1 : U_2] +  E_{U_1}\bigl[D[(X | U)_1 : (X | U)_2] \bigr],
\end{equation}
where $E_{U_1}[\cdot ]$ is the expectation under the condition $U \stackrel{d}{=} U_1$.
(For more details on divergences, see Vajda \cite{Vajda} and Amari and Nagaoka \cite{Amari&Nagaoka}.) 

 Thanks to this property, we can decompose the MLE risk into several parts, each of which corresponds to the first- and second-stage distributions. 

When a sample of size $n$ is taken from the two-stage multinomial distribution, let $x_{ij}$ and $ x_{i\cdot}$  $(i=1,\ldots,I, j=1,\ldots,J_i)$ denote the number of individuals that belong to $C_{ij}$ and $C_{i\cdot}$, respectively, hence
\[
x_{i\cdot} = \sum_{j=1}^{J_i}x_{ij},\qquad \sum_{i=1}^{I}x_{i\cdot}=n.
\]
The notations of $X_{ij}$ and $X_{i \cdot}$ are similarly defined as the corresponding random variables.

In this section, we always assume that $m_{i\cdot}$ and $p_{ij}$ are parameterized independently, that is, 
\begin{align}
&\text{1. $m_f\triangleq (m_{1\cdot},\ldots,m_{I\cdot}) $ depends on $\tau\triangleq(\tau_1,\ldots,\tau_t)$.} \label{parameter_first_stage}\\
&\text{2. For each $i\ (1\leq i \leq I)$, $p_i\triangleq (p_{i1},\ldots, p_{iJ_i})$ depends on 
$\theta^{(i)}\triangleq (\theta_1^{(i)},\ldots,\theta_{s_i}^{(i)})$.} \label{parameter_second_stage}
\end{align}
%
%
% Theorem 1
%
%
We obtain the following result.
\begin{theo}
\label{MLE_2nd_stage}
The MLE for the two-stage multinomial distribution model, $\widehat{m}=(\widehat{m}_{ij})$, is given by
\begin{equation}
\label{decomp_estimator_2stage}
\widehat{m}_{ij}=\widehat{m}_{i\cdot}\widehat{p}_{ij},\qquad i=1,\ldots,I, \ j=1,\ldots, J_i,
\end{equation}
with 
\[
\widehat{m}_{i\cdot}=m_{i\cdot}(\widehat{\tau}),\quad
\widehat{p}_{ij}=p_{ij}(\widehat{\theta}^{(i)}),\qquad i=1,\ldots,I, \ j=1,\ldots, J_i,
\]
where $\widehat{\tau}$ is the MLE for the first-stage model based on $x_{1\cdot},\ldots,x_{I\cdot}$, and for each $i(i=1,\ldots,I)$, $\widehat{\theta}^{(i)}$ is the MLE for the second-stage model based on $x_{i1},\ldots,x_{iJ_i}$.
\end{theo}
\begin{proof}
For the sample $x=(x_{ij})$, the log-likelihood is expressed as 
\begin{align*}
\log f(x;m)&=\sum_{i=1}^I \sum_{j=1}^{J_i} x_{ij} \log m_{ij}+c(x) \\
&=\sum_{i=1}^I \Bigl( \sum_{j=1}^{J_i}x_{ij}  \bigl(\log m_{ij}/m_{i\cdot}+\log m_{i\cdot}\bigr)\Bigr)+c(x) \\
&=\sum_{i=1}^I \Bigl( \sum_{j=1}^{J_i}x_{ij}  \bigl(\log p_{ij}(\theta^{(i)})+\log m_{i\cdot}(\tau)\bigr)\Bigr)+c(x) \\
&=\sum_{i=1}^I x_{i\cdot} \log m_{i\cdot}(\tau)+\sum_{i=1}^I \sum_{j=1}^{J_i}x_{ij}  \log p_{ij}(\theta^{(i)})+c(x) \\
&\Bigl(\triangleq \log f(x; \tau,\theta), \quad \theta=(\theta^{(1)},\ldots,\theta^{(I)})\Bigr)
\end{align*}
We obtain 
\[
\max_{\tau,\theta}f(x; \tau,\theta)=\max_{\tau} \sum_{i=1}^I x_{i\cdot} \log m_{i\cdot}(\tau)+\sum_{i=1}^I \max_{\theta^{(i)}}\sum_{j=1}^{J_i}x_{ij}  \log p_{ij}(\theta^{(i)}).
\]
We notice that
\[
\sum_{i=1}^I x_{i\cdot} \log m_{i\cdot}(\tau)
\]
is the log-likelihood function of $\tau$ up to constant for the first-stage model based on $x_{1\cdot},\ldots,x_{I\cdot}$, whereas, for each $i (1\leq i \leq I)$, the log-likelihood function of $\theta^{(i)}$ for the second-stage model is given by 
\[
\sum_{j=1}^{J_i}x_{ij}  \log p_{ij}(\theta^{(i)})
\]
up to constant with the sample $x_{i1},\ldots,x_{iJ_i}$. Since $m_{ij}=m_{i\cdot}(\tau)p_{ij}(\theta^{(i)})$, all the results are obtained.
\end{proof}

%Hereon, we consider only MLE as the estimator $\widehat{m}$ of $m$, which is given by \eqref{decomp_estimator_2stage}. 

The following decomposition of $D[\widehat{m}: m]$ holds;
\begin{align}
D[\widehat{m}: m]&=\sum_{i=1}^I \sum_{j=1}^{J_i} \widehat{m}_{ij} \log (\widehat{m}_{ij}/m_{ij})\nonumber\\
&=\sum_{i=1}^I \sum_{j=1}^{J_i} \widehat{m}_{i\cdot}\widehat{p}_{ij} \log \frac{\widehat{m}_{i\cdot}\widehat{p}_{ij}}{m_{i\cdot}p_{ij}}\nonumber\\
&=\sum_{i=1}^I \widehat{m}_{i\cdot} \log (\widehat{m}_{i\cdot}/m_{i\cdot}) +\sum_{i=1}^I \widehat{m}_{i\cdot}\sum_{j=1}^{J_i}\widehat{p}_{ij} \log(\widehat{p}_{ij}/p_{ij})\nonumber\\
&=D[\widehat{m}_f: m_f] + \sum_{i=1}^I \widehat{m}_{i\cdot} \:D[\widehat{p}_i:p_i], \label{chain_rule_-1}
\end{align}
where 
\[
\widehat{m}_f= (\widehat{m}_{1\cdot},\ldots\widehat{m}_{I\cdot}),\qquad \widehat{p}_i\triangleq (\widehat{p}_{i1},\ldots,\widehat{p}_{i J_i}),\quad 1\leq i \leq I.
\]

When the first-stage model is full, the MLE of $m_{i\cdot}$ is given by
$\widehat{m}_{i\cdot}=X_{i\cdot}/n,\ i=1,\ldots,I$, and the asymptotic expansion of the risk of the MLE for the first-stage model (denoted by $ED_f$) with the sample size $n$ is given by (see \eqref{ED_-1_expan_full})
\begin{equation}
\label{ED_-1_first_stage}
ED_f=\frac{I-1}{2n}+\frac{1}{12n^2}(M_f-1)+o(n^{-2}),
\end{equation}
where
\begin{equation}
\label{def_M_f}
M_f\triangleq \sum_{i=1}^I m_{i\cdot}^{-1}.
\end{equation}
For the $i$-th model ($i=1,\ldots,I$) at the second stage (not necessarily a full model), let $\overset{-1}{A}_{(i)}/(24n^2)$ denote the second-order term of $ED_{(i)}\triangleq E[D[\widehat{p}_i:p_i]]$ with the sample size $n$.
Namely, we have 
\begin{equation}
\label{i_th_second_stage_expan}
ED_{(i)}=\frac{s_i}{2n}+\frac{1}{24n^2}\overset{-1}{A}_{(i)}(p_i)+o(n^{-2}).
\end{equation}
Because the dimension of the $i$-th model at the second stage equals $s_i$ due to \eqref{parameter_second_stage}, the first-order term equals $s_i/(2n)$ (see
Theorem 1 of \cite{Sheena}).

It should be noted that it is possible that $X_{i \cdot}=0$. Hence, we can encounter a situation in which we are unable to estimate $p_{ij},\ j=1,\ldots,J_i$, the parameter of the $i$-th second-stage model, because there is no available sample. We overcome this problem by making it a rule to discard such samples with no estimation. In the following theorems, all the expectations are conditional on the state $X_{i \cdot} \ne 0,\ 1 \leq \forall i \leq I$. However, as Lemma \ref{expectations_difference} in the Appendix shows, the conditional expectations of Kullback--Leibler divergence and $(X_{i\cdot}-n m_{i \cdot})^k (k=1,2,\ldots)$ differ from those that are unconditional by $o(n^{-s})$ for any $s >0$. Therefore, we use the same notation as that for the unconditional distribution.(Avoiding zero probability estimates is a practically important issue. See Darsheid et.al. \cite{Darsheid_et_al}.)

The next theorem is Theorem 1 of Sheena \cite{Sheena_4}. It provides the decomposition of $ED$ when the first-stage model is a full model. 
%
%
%
% Theorem 2
%
%
%
\begin{theo}
\label{ED_-1_2stage}
If the first-stage model is a full model, then the risk of MLE is equal to
\begin{equation}
\label{risk_full_model}
ED[\widehat{m}:m]= \frac{p'}{2n}+\frac{1}{24n^2}\Bigl\{\sum_{i=1}^I m_{i\cdot}^{-1}\Bigl(\overset{-1}{A}_{(i)}+2\Bigr)-2\Bigr\}+o(n^{-2}), 
\end{equation}
where 
\[
p'\triangleq I-1+\sum_{i=1}^I s_i.
\]
\end{theo}
%
%
%
%
%
%If, besides the conditions of Theorem \ref{ED_-1_2stage}, the second-stage models are also full models, then the total model is a full model. We will see the right-hand side of \eqref{ED_-1_expan_two_stage} for this case. We notice that
%\[
%s_i=J_i-1,\quad \overset{-1}{A}_{(i)}= 2(M_{(i)}-1)\ \text{(see %\eqref{ED_-1_expan_full})},\qquad i=1,\ldots,I,
%\]
%where 
%\[
% M_{(i)}\triangleq \sum_{j=1}^{J_i} m_{i\cdot}/m_{ij}.
%\]
%Let $p\triangleq \sum_{i=1}^I J_i-1$, then we have $p'=I-1+\sum_{i=1}^I (J_i-1)=p$ and
%\[
%\sum_{i=1}^I m_{i\cdot}^{-1}\Bigl(\overset{-1}{A}_{(i)}+2\Bigr)=2\sum_{i=1}^I m_{i\cdot}^{-1}M_{(i)}=2\sum_{i=1}^I\sum_{j=1}^{J_i} m_{ij}^{-1},\quad i=1,\ldots,I.
%\]
%Finally we have 
%\[
%ED=\frac{p}{2n}+\frac{1}{12n^2}\Bigl\{\sum_{i=1}^I \sum_{j=1}^{J_i} m_{ij}^{-1}-1\Bigr\}+o(n^{-2}),
%\]
%which is equivalent to \eqref{ED_-1_expan_full} as is expected.
%
%
%
%
%
\subsection{Asymptotic risk expansion for the submodel}
%
%
% Theorem 3
%
%
\ \ Suppose that we have prior knowledge on $\tau$ in \eqref{parameter_first_stage}, equivalently $m_f$. The next theorem gives the risk of MLE for this situation.
\begin{theo}
\label{ED_-1_2stage_fst_known}
If the first-stage model is a full model and $\tau_i\ (i=1,\ldots,I)$ are all known, then the risk of MLE is
\begin{equation}
\label{risk_sub_model}
ED[\widehat{m}:m]
=\frac{1}{2n}\sum_{i=1}^I s_i +\frac{1}{24n^2}\sum_{i=1}^I m_{i\cdot}^{-1}\Bigl(\overset{-1}{A}_{(i)}+12(1-m_{i\cdot})s_i\Bigr)+o(n^{-2}).
\end{equation}
Especially, if all the second-stage models are full models, then the risk of MLE is 
\begin{equation}
\label{risk_sub_2nd_full} 
ED[\widehat{m}:m]
=\frac{p'-(I-1)}{2n} +\frac{1}{12n^2}\Bigl(\sum_{i=1}^I\sum_{j=1}^{J_i} m_{ij}^{-1}+\sum_{i=1}^I m_{i\cdot}^{-1}(6J_i-7)-6(p'+1-I)\Bigr)+o(n^{-2}).
\end{equation}
\end{theo}
\begin{proof}
Since $\widehat{m}_{i\cdot}=m_{i\cdot},\ i=1,\ldots,I$, that is, $\widehat{m}_f=m_f$, we have 
\[
D[\widehat{m}: m]= \sum_{i=1}^I m_{i\cdot} \:D[\widehat{p}_i:p_i]
\]
(see \eqref{chain_rule_-1}). 
\begin{align*}
&ED[\widehat{m}:m]\\
&\triangleq E[D[\widehat{m}:m]]\\
&=E_{X_{1\cdot},\cdots,X_{I\cdot}} \bigl[E[D[\widehat{m}:m] | X_{1\cdot},\ldots,X_{I\cdot}] \bigr]\\
&= E_{X_{1\cdot},\cdots,X_{I\cdot}} \Bigl[\sum_{i=1}^I m_{i\cdot} \:E\bigl[D[\widehat{p}_i:p_i] \big| X_{1\cdot},\ldots,X_{I\cdot}\bigr] \Bigr]\\
&=\sum_{i=1}^I  E_{X_{i\cdot}} [ m_{i\cdot} \:ED_{(i)}(X_{i\cdot}) ],
\end{align*}
where
\begin{align*}
ED_{(i)}(X_{i\cdot})&\triangleq  E[D\bigl[\widehat{p}_{i}:p_{i}]\big|X_{i\cdot}\bigr].
\end{align*}
From \eqref{i_th_second_stage_expan}, substituting $n$ with $X_i$, we have
\[
ED_{(i)}(X_{i\cdot})=\frac{s_i}{2X_{i\cdot}}+\frac{1}{24X_{i\cdot}^2}\overset{-1}{A}_{(i)}+r_i(X_{i\cdot}),
\]
where 
\[
|X_{i\cdot}^3 r_i(X_i)| < \exists M_i.
\]
Therefore we have
\begin{align*}
&E_{X_{i\cdot}} [ m_{i\cdot} \:ED_{(i)}(X_{i\cdot}) ]\\
&=\frac{s_i m_{i\cdot}}{2}E[X_{i\cdot}^{-1}]+\frac{m_{i\cdot}}{24}\overset{-1}{A}_{(i)}E[X_{i\cdot}^{-2}]+E[r_i(X_{i\cdot})].
\end{align*}

Taylor expansions of $1/X_{i\cdot}$ and $1/X^2_{i\cdot}$ are given by
\begin{align*}
\frac{1}{X_{i \cdot}}&=\frac{1}{nm_{i\cdot}}-\frac{1}{(nm_{i\cdot})^2}(X_{i \cdot}-nm_{i\cdot})+\frac{1}{(nm_{i\cdot})^3}(X_{i \cdot}-nm_{i\cdot}) ^2+ \frac{1}{(nm^*_{i\cdot})^4}(X_{i \cdot}-nm_{i\cdot}) ^3 \\
\frac{1}{X^2_{i \cdot}}&=\frac{1}{(nm_{i\cdot})^2}-\frac{2}{(nm^*_{i\cdot})^3}(X_{i \cdot}-nm_{i\cdot})\\
\frac{1}{X^3_{i \cdot}}&=\frac{1}{(nm_{i\cdot})^3}-\frac{3}{(nm^*_{i\cdot})^4}(X_{i \cdot}-nm_{i\cdot})\\
\end{align*}
with suitable choice of $m^*_{i\cdot}\leq 1$ for each expansion.
Using $E[X_{i \cdot}-nm_{i\cdot}]=0$ and $E[(X_{i \cdot}-nm_{i\cdot})^2]=nm_{i\cdot}(1-m_{i\cdot})$ and the asymptotic distribution 
\[
\frac{1}{\sqrt{n}}(X_{i \cdot}-nm_{i\cdot}) \stackrel{d}{\rightarrow} N(0, m_i(1-m_i))
\]
we have
\begin{align}
\label{expec_X^-1}
E[X_{i \cdot}^{-1}]&=\frac{1}{nm_{i\cdot}}+\frac{1-m_{i\cdot}}{(nm_{i\cdot})^2}+O(n^{-5/2})\\
\label{expec_X^-2}
E[X_{i \cdot}^{-2}]&=\frac{1}{(nm_{i\cdot})^2}+O(n^{-5/2})\\
\label{expec_X^-3}
E[X_{i \cdot}^{-3}]&=\frac{1}{(nm_{i\cdot})^3}+O(n^{-7/2})
\end{align}
and 
\[
E[|r_i(X_{i\cdot}|]=E[X_{i\cdot}^{-3}|X_{i\cdot}^3r_i(X_{i\cdot})|] \leq M_i E[X_{i \cdot}^{-3}]= O(n^{-3}).
\]
From these evaluations, we have
\[
E_{X_{i\cdot}} [ m_{i\cdot} \:ED_{(i)}(X_{i\cdot}) ]=\frac
{s_i}{2n}+\frac{1}{24n^2}m_{i\cdot}^{-1}\Bigl(\overset{-1}{A}_{(i)}+12(1-m_{i\cdot})s_i\Bigr)+o(n^{-2}).
\]
Consequently, we have 
\begin{equation}
\label{ED_-1_two_stage_b}
ED[\widehat{m}:m]=\sum_{i=1}^I \Bigl(\frac
{s_i}{2n}+\frac{1}{24n^2}m_{i\cdot}^{-1}\Bigl(\overset{-1}{A}_{(i)}+12(1-m_{i\cdot})s_i\Bigr)\Bigr)+o(n^{-2}).
\end{equation}
When all the second-stage models are full, 
\[
s_i = J_i-1, \qquad \overset{-1}{A}_{(i)}= 2\Bigl(m_{i\cdot}\sum_{j=1}^{J_{i}}m_{ij}^{-1}-1\Bigr),\qquad i=1,\ldots,I.
\]
If we insert these results into \eqref{risk_sub_model}, we have
\[
\sum_{i=1}^I s_i =\sum_{i=1}^I (J_i-1)=p'+1-I
\]
and 
\begin{align*}
\sum_{i=1}^I m_{i\cdot}^{-1}\Bigl(\overset{-1}{A}_{(i)}+12(1-m_{i\cdot})s_i\Bigr)
&=\sum_{i=1}^I m_{i\cdot}^{-1}\Bigl(2 m_{i\cdot}\sum_{j=1}^{J_{i}}m_{ij}^{-1}-2+12(1-m_{i\cdot})(J_i-1)\Bigr)\\
&=2\sum_{i=1}^I \sum_{j=1}^{J_i}m_{ij}^{-1}+12\sum_{i=1}^I(1-J_i)+\sum_{i=1}^I m_{i\cdot}^{-1}(12J_i-14)\\
&=2\sum_{i=1}^I \sum_{j=1}^{J_i}m_{ij}^{-1}+2\sum_{i=1}^I m_{i\cdot}^{-1}(6J_i-7)+12(I-p'-1).
\end{align*}
\end{proof}
The prior knowledge on the values of $m_f$ is equivalent to the $m$-aggregation submodel
\begin{equation}
\label{non-overlap_m_sub}
m_{i \cdot}\triangleq\sum_{j=1}^{J_i} m_{ij}=c_i,\qquad i=1,\ldots,I
\end{equation}
where $c_i$'s are known constants. It is clear that any non-overlapping $m$-aggregation submodel can be treated as a two-stage model in which every second-stage model is a full model. Consequently, the risk for the non-overlapping $m$-aggregation submodel is given by \eqref{risk_sub_2nd_full}. 

We now compare the risk of \eqref{risk_full_model} and that of the \eqref{risk_sub_model}. Notice that the difference between Theorem \ref{ED_-1_2stage} and Theorem \ref{ED_-1_2stage_fst_known} is the existence of the prior information on $M_f$. We call the models in Theorem \ref{ED_-1_2stage} and \ref{ED_-1_2stage_fst_known} ``M1'' and ``M'2,'' respectively. If the second-stage models are all full, then M1 becomes the full model and M2 becomes the $m$-aggregation submodel. 

If we neglect the $o(n^{-2})$ terms, the difference between \eqref{risk_full_model} and \eqref{risk_sub_model} is equal to
\begin{equation}
\label{diff_risk_full_sub}
\frac{I-1}{2n}+\frac{1}{12n^2}\Bigl(M_f-1-6\sum_{i=1}^I (m_{i\cdot}^{-1}-1)s_i \Bigr),
\end{equation}
especially, when the second-stage models are all full models, 
\begin{equation}
\label{diff_risk_full_sub_2}
\frac{I-1}{2n} + \frac{1}{12n^2}\Bigl(-\sum_{i=1}^I m_{i\cdot}^{-1}(6J_i-7)+6(p'-I)+5\Bigr).
\end{equation}
Regarding the first-order terms, that of M2 is always smaller than that of M1. The risk ratio between the full and submodels is close to the dimension ratio 
\[
\frac{\sum_{i=1}^I s_i }{ I-1+\sum_{i=1}^I s_i}
\]
under a large enough sample size.

However, the second-order term of M2 can be larger than that of M1. Suppose that $s_i \geq 1 (i=1,\ldots,I)$, then 
\begin{align*}
M_f-1-6\sum_{i=1}^I (m_{i\cdot}^{-1}-1)s_i & \leq M_f-1-6\sum_{i=1}^I (m_{i\cdot}^{-1}-1) \\
&=-5M_f+6I-1 < -5 I^2 +6I -1 =-(5I-1)(I-1) < 0.
\end{align*}
the second inequality holds due to the rule that the arithmetic mean is not less than the harmonic mean
\begin{equation}
\label{lower_Mf}
\frac{M_f}{I}=\frac{\sum_{i=1}^I m_{i \cdot}^{-1}}{I} \geq \frac{I}{\sum_{i=1}^I m_{i\cdot}}=I.
\end{equation}

When the second-stage models are all full, $s_i = 0$ is equivalent to $J_i = 1$, namely 
\[
m_{i\cdot}=m_{i1}=c_i
\]
in \eqref{non-overlap_m_sub}. We call the restriction equation of this type ``solid'' as it provides solid information on the probability of a particular cell. The aforementioned result is applied to the $m$-aggregation model as a corollary.
\begin{coro}
\label{coro_second_order_term}
If none of the restriction equations of the $m$-aggregation submodel are solid, then the second-order term of the submodel is larger than that of the full model.
\end{coro}

This means that \eqref{diff_risk_full_sub_2} is negative for some small values of $n$ when the restriction equations are all non-solid. To simplify the explanation, we roughly classify the sample size as follows: ``small sample size'' when \eqref{diff_risk_full_sub_2} is negative; ``medium sample size'' when \eqref{diff_risk_full_sub_2} is positive, but the second-order term of \eqref{diff_risk_full_sub_2} is still non-negligible; ``large sample size'' when the second-order term of \eqref{diff_risk_full_sub_2} is negligible.

Our conclusion is as follows:
\begin{itemize}
\item When $n$ is ``small,'' the rather pathological situation may occur that the submodel (i.e., the prior information) increases the risk, and hence, it is better not to use the submodel. 
\item When $n$ is ``medium,'' the submodel has an advantage over the full model. However, the submodel loses the estimation efficiency gained by the dimension reduction due to the larger second-order term. That is, the ratio of the risks between the sub and full models is not as small as the dimension ratio. 
\item When $n$ is ``large,'' the risk ratio between the sub and full models is close to the dimension ratio. It should be noted that the dimension ratio is close to one when the number of restriction equations in the submodel ($=I$) is quite small compared with the dimension of the full model ($=\sum_{i=1}^I J_i -1$). 
%\item Generally speaking, the use of the submodel (prior knowledge) does not make much contribution when $s_i$'s are large (i.e. the second-stage models are complicated) and (or) $M_f$ is large (i.e. $m_{i\cdot}$ of small value exists).
\end{itemize}

%The comparison between the full and sub models is based on \eqref{risk_full_model} and \eqref{risk_sub_model}, which are ``approximations'' of the true risks obtained by neglecting $o(n^{-2})$ terms. 
%
%
%
%
%
%
%
\section{Numeric analysis of some examples}
\label{example}
\ \ In the previous subsection, we observed the risk difference between the full model and the non-overlapping $m$-aggregation submodel through the approximated risks, that is, the asymptotic expansion of the risks up to the second order. In this subsection we confirm several results using simulation with three examples. Every example is a two-way contingency table and the non-overlapping $m$-aggregation submodel is given as the one in which the column sums are all known.  As a two-stage model, the first-stage model consists of the multinomial distribution over the columns and each second-stage model is the distribution over the rows within a given column.

To compare the submodel and full model, we also use the indicator, ``the required sample size (r.s.s) of the submodel to the full model under the condition $n_0$,'' which is defined by the solution $n^*$ of the following equation 
\[
ED _{s}(n^*) = ED _f(n_0),
\]
where $ED _{s}(n)$ and $ED _{f}(n)$ are the risks of the sub and full model considered as the function of the sample size, respectively. This reveals the required sample size for the submodel is equal in terms of risk with the full model of the sample size $n_0$. The risks in the equation above are calculated by approximation or simulation.

The following abbreviations are commonly used in the three examples;\\
``f.risk.sim(app)'' is the risk of the full model obtained by simulation(approximation).\\
``s.risk.sim(app)'' is the risk of the submodel obtained by simulation(approximation).\\
``ratio.sim(app)'' is the risk ratio between the submodel and the full model based on the simulated(approximated) risks.\\
``r.s.s.sim(app)'' is the r.s.s obtained by simulation(approximation).
%
%
%
% Example 1  100 by 2 and 20 by 10 tables
%
%
\bigskip
\\
--\textit{Example 1}-- \\
The first example is an artificial setting to confirm some theoretical results in the previous subsection.  We consider a 100 by 2 contingency table for the two cases of contrasting distributions. For both cases, $ I=2, \ J_1=J_2=100,$ where in Example 1-1,
\[
m_{ij}=1/200\ (i=1,2,\  j=1,\ldots,100),\quad m_{1\cdot}=m_{2\cdot}=1/2.
\]
and in Example 1-2,
\[
m_{1j}=1/2000,\ m_{2j}=9/2000,\ j=1,\ldots,200.
\]
These distributions  are chosen from the following consideration;  Since the second-order term in the equation \eqref{diff_risk_full_sub_2} equals 
\[
\frac{1}{12n^2}\Bigl(- 593(m_{1\cdot}^{-1}+m_{2\cdot}^{-1})+6p'-7\Bigr),
\]
which is maximized when  $m_{1\cdot}=m_{2\cdot}=1/2$ (Example 1-1) and gets smaller when $m_{1\cdot}=1/10,\ m_{2\cdot}=9/10$ (Example 1-2).

The risks and r.s.s. under several values of $n$ are given in Table \ref{Table: risk_ex1-1} for Example 1-1 and Table \ref{Table: risk_ex1-2} for Example 1-2. The number in the parenthesis after the simulated risk is the standard deviation. The r.s.s is under the condition $n_0=n$. For the calculation of the simulated risk, we used $10^4$ sets of samples and took the average over these. 
\begin{table}
\caption{Example 1-1---Risk and R.S.S.---}
\label{Table: risk_ex1-1}
\centering
\scriptsize
\begin{tabular}[t]{|c|c|c|c|c|c|c|c|c|}
\hline
n & f.risk.app & s.risk.app & ratio.app & r.s.s.app & f.risk.sim & s.risk.sim &  ratio.sim & r.s.s.sim \\ \hline
100 & 1.3283  & 1.3332  & 1.0037  & 100  & 1.004589(5.3e-6)  & 1.006799(5.3e-6)  & 1.0022  & 101  \\ \hline
200 & 0.5808  & 0.5808  & 0.9999  & 200  & 0.571158(4.1e-6)  & 0.571305(4.1e-6)  & 1.0003  & 201  \\
\hline
300 & 0.3687  & 0.3681  & 0.9985  & 300  & 0.384620(3.3e-6)  & 0.384268(3.3e-6)   & 0.9991  & 300  \\
\hline
400 & 0.2696  & 0.2690  & 0.9977  & 399  & 0.283346(2.6e-6)  & 0.282871(2.6e-6)  & 0.9983  & 400  \\
\hline
500 & 0.2123  & 0.2117  & 0.9972  & 499  &  0.221933(2.2e-6)  & 0.221424(2.2e-6)   & 0.9977  & 499  \\
\hline
1000 & 0.1028  & 0.1024  & 0.9961  & 996  & 0.104093(1.1e-6)  & 0.103697(1.1e-6)  & 0.9962  & 996  \\ \hline
2000 & 0.0506  & 0.0504  & 0.9955  & 1991  & 0.050657(5.1e-7)   & 0.050430(5.1e-7)   & 0.9955  & 1992  \\ \hline
\end{tabular}
\end{table}
\begin{table}
\caption{Example 1-2---Risk and R.S.S.---}
\label{Table: risk_ex1-2}
\centering
\scriptsize
\begin{tabular}[t]{|c|c|c|c|c|c|c|c|c|}
\hline
n & f.risk.app & s.risk.app & ratio.app & r.s.s.app & f.risk.sim & s.risk.sim &  ratio.sim & r.s.s.sim \\ \hline
100 & 1.9209  & 1.9609  & 1.0208  & 101  & 0.800707(6.9e-6)  & 0.804596(5.7e-6)  & 1.0049  & 101 \\ \hline
200 & 0.7290  & 0.7377  & 1.0120  & 202  & 0.458571(4.1e-6)  & 0.460303(3.9e-6)  & 1.0038  & 201  \\
\hline
300 & 0.4345  & 0.4379  & 1.0077  & 302  & 0.322643(2.9e-6)  & 0.323684(2.9e-6)   & 1.0032  & 301  \\
\hline
400 & 0.3066  & 0.3082  & 1.0051  & 402  & 0.250190(2.1e-6)  & 0.250864(2.3e-6)  & 1.0027  & 402  \\
\hline
500 & 0.2360  & 0.2368  & 1.0034 & 501  &  0.205073(1.7e-6)  & 0.205388(1.9e-6)   & 1.0015  & 502  \\
\hline
1000 & 0.1088  & 0.1087  & 0.9995  & 1000  & 0.107846(9,2e-7)  & 0.107870(1.0e-6)  & 1.0002  & 1000  \\ \hline
2000 & 0.0521  & 0.0519  & 0.9974  & 1955  & 0.053458(5.1e-7)   & 0.053359(5.5e-7)   & 0.9981 & 1997  \\ \hline
\end{tabular}
\end{table}
We observe the following results:
\begin{itemize}
\item For $n=100\text{ or } 200$ in Ex.1-1 and all $n$'s except $n=2000$ in Ex.1-2,  the use of the submodel increases the risk. 
\item Since the dimension ratio ($=198/199\doteqdot0.9950 $) is close to one, the use of the submodel causes little risk reduction. In view of r.s.s., the contribution of the submodel is almost negligible. 
\item The approximated values for the risk are somewhat deviated from the simulated values when $n$ is small; however, those for the risk ratio or r.s.s. are quite close to the simulated values even when $n$ is small. 
\end{itemize} 
%
% 
%
% Example 2 Breast cancer classification
%
%
--\textit{Example 2}-- \\
We use real data on breast cancer taken from the "UCI machine learning repository" (https://archive.ics.uci.edu/ml/datasets/Breast+Cancer). We created a cross-tabulation table from the variable ``the degree of malignancy'' (3 levels:1,2,3) and ``age group'' (5 groups: 30--39,...,70--79) excluding one person in his or her twenties from the original data set. We gained the relative frequency by dividing each cell by the total number of individuals  ,285(see Table \ref{table:breast_cancer_prob}). We assume that this is the true probability $m_{ij},\ i=1,\ldots,3, j=1,\ldots,5$ for each category. We consider the submodel according to the situation in which we have prior knowledge of the sums of each column, that is, the distribution over the age groups.

Table \ref{table:breast_cancer_risk_full_sub} reveals the approximated and simulated risk (in parentheses) for the full and submodels for several cases of sample size $n$. It also shows the approximated and simulated r.s.s. of the submodel to the full model under the condition $n_0=n$ (the simulated r.s.s is in parentheses). 

We summarize the result as follows:
\begin{itemize}
\item For this example, when $n\leq 28$, the value of \eqref{diff_risk_full_sub} (equivalently \eqref{diff_risk_full_sub_2}) becomes negative. However, the simulation reveals that the risk of the submodel is always smaller than that of the full model. In fact, even if $n=10$, the former equals 0.6236 and the latter equals 0.6543.  
\item The dimension ratio is $10/14\doteqdot 0.714$. Since the table is small, the knowledge of the column sums significantly reduces the risk. However, we notice that the effect of the dimension reduction is lessened by the larger second-order term of the submodel as the ``Risk Ratio'' or ``R.S.S./$n$'' is larger than 0.714.
\item The approximation method for risk calculation is effective under the sample sizes in the table.
\end{itemize}
\begin{table}
\caption{Breast cancer classification} 
\label{table:breast_cancer_prob}
\centering
\begin{tabular}{|c | c | c | c | c | c| }
\hline
   &  30--39  &  40--49  &  50--59  &  60--69  &  70--79 \\
\hline
1 &0.025 & 0.063 & 0.088 & 0.060 & 0.014 \\
\hline
2 &0.060 & 0.168 & 0.137 & 0.084 & 0.004 \\
\hline
3 &0.042 & 0.084 & 0.112 & 0.056 & 0.004 \\
\hline
Column sum & 0.126 & 0.317 & 0.337 & 0.200 & 0.021\\
\hline
\end{tabular}
\end{table}
\begin{table}
\caption{Example 2---Risk and R.S.S.--- }
\label{table:breast_cancer_risk_full_sub}
\footnotesize
\begin{tabular}{c | c | c | c | c | c }
   &  $n=200$ &  $n=400$  &  $n=600$  &  $n=800$  &  $n=1000$ \\
\hline
Full Model & 0.0367(0.0361) &0.0179(0.0180) &0.0119(0.0119) &0.0089(0.0090) &0.0071(0.0071)
\\
\hline
Submodel & 0.0281(0.0277) &0.0133(0.0135) &0.0087(0.0088) &0.0064(0.0066) &0.0051(0.0052)\\
\hline
Risk Ratio &  0.766(0.769) & 0.741(0.752) &0.732(0.741) & 0.728(0.734) & 0.725(0.730) \\
\hline
R.S.S. & 158(155) & 302(305) &445(450) &588(591) &732(739)\\
\hline
R.S.S./$n$ & 0.790(0.775) &0.755(0.762)  &0.742(0.750) &0.735(0.739) &0.732(0.739)\\
\end{tabular}
\end{table}
%
%
%
%  Example 3  Household classification
%
%
\bigskip
--\textit{Example 3}-- \\
We use data from the ``2014 National Survey of Family Income and Expenditure'' by the Statistics Bureau in Japan (https://www.stat.go.jp/english/data/zensho/index.html). Table \nolinebreak \ref{table:prob_income_age} is the relative frequency obtained from the classification of 100,006 households according to ``Yearly income group (Y1,...,Y10)'' and ``Household age group (H1,...,H6).'' We use this relative frequency as the population parameter $m_{ij},\ 1\leq i \leq 10,\ 1\leq j \leq 6$. We consider the submodel that is based on the prior knowledge of all the column sums. Table \ref{table:household_risk_full_sub} presents the results. 

We summarize the results as follows:
\begin{itemize}
\item For this example, when $n\leq 133$, the value of \eqref{diff_risk_full_sub} (equivalently \eqref{diff_risk_full_sub_2}) becomes negative. The simulation result (not presented in Table \ref{table:household_risk_full_sub} ) shows that the pathological phenomenon (the reversed order in the risks) actually occurs, but only when $n$ is as small as 20. 
\item The dimension ratio is $54/59\doteqdot 0.915$. The table is larger than that of Example 2, and the knowledge of the column sums is not as useful in terms of reducing the dimension. The larger second-order term of the submodel is also a burden on risk reduction, which increases the ``Risk Ratio'' or ``R.S.S./$n$'' to over 0.915.
\item The approximation method for risk calculation works effectively under the sample sizes in the table.
\end{itemize}
\begin{table}
\caption{Household classification}
\label{table:prob_income_age}
\centering
\footnotesize
\begin{tabular}{|c|c|c|c|c|c|c|}
\hline
    &  H1 & H2 & H3 & H4 & H5 & H6 \\ \hline
Y1 & 0.00161 & 0.00232 & 0.00512 & 0.00395 & 0.00468 & 0.00066  \\ \hline
Y2 & 0.00331 & 0.0081 & 0.00953 & 0.00783 & 0.01145 & 0.00278 \\ \hline
Y3 &0.00974 & 0.02109 & 0.02046 & 0.01499 & 0.02536 & 0.00494 \\ \hline
Y4 &0.00799 & 0.03519 & 0.03229 & 0.02017 & 0.0338 & 0.00708  \\ \hline
Y5 &0.00547 & 0.0376 & 0.04362 & 0.02442 & 0.02675 & 0.00398 \\ \hline
Y6 &0.00494 & 0.05082 & 0.09003 & 0.05772 & 0.03732 & 0.00452  \\ \hline
Y7 &0.00126 & 0.02106 & 0.0543 & 0.05531 & 0.01999 & 0.00234 \\ \hline
Y8 &0.00071 & 0.00961 & 0.0323 & 0.04043 & 0.0108 & 0.00122  \\ \hline
Y9 &0.00011 & 0.00201 & 0.01204 & 0.02184 & 0.00466 & 0.00052  \\ \hline
Y10 & 0.00006 & 0.00139 & 0.00697 & 0.01582 & 0.00344 & 0.00022 \\ \hline
Col. Sum & 0.03520 & 0.18919 & 0.30666 & 0.26248 & 0.17825 & 0.02826 \\ \hline
\end{tabular}
\end{table}
\begin{table}
\caption{Example 3---Risk and R.S.S.---}
\label{table:household_risk_full_sub}
\footnotesize
\begin{tabular}{c | c | c | c | c | c }
   &  $n=1000$ &  $n=1500$  &  $n=2000$  &  $n=2500$  &  $n=3000$ \\
\hline
Full Model & 0.0372(0.0295) &0.0231(0.0197) &0.0167(0.0148) &0.0130(0.0119) &0.0107(0.0099)
\\
\hline
Submodel & 0.0351(0.0275) &0.0216(0.0183) &0.0155(0.0137) &0.0121(0.0109) &0.0099(0.0091)\\
\hline
Risk Ratio &  0.945(0.931) & 0.937(0.926) &0.932(0.924) & 0.929(0.923) & 0.927(0.922)\\
\hline
R.S.S. & 955(934) & 1419(1392) &1880(1848) &2339(2305) &2799(2773)\\
\hline
R.S.S./$n$ & 0.955(0.934) &0.946(0.928)  &0.940(0.924) &0.936(0.922) &0.933(0.924)\\
\end{tabular}
\end{table}
\section{Submodel and Divergence}
%The theoretical and numerical analyses revealed the following points:
%\begin{itemize}
%\item Risk reduction through the use of the submodel (prior information) is mainly determined by the dimension reduction. If we have many restriction equations and a large sample, the submodel significantly reduces the risk. 
%\item When we have relatively few restriction equations compared with the dimension of the full model, and the sample size is small, the $m$-aggregation submodel can increase the risk. We are tempted to use the submodel as ``compensation'' for the small sample, but it can increase the risk. The equation \eqref{diff_risk_full_sub} (or \eqref{diff_risk_full_sub_2}) is useful to check if the estimation condition (the dimension of the submodel,  the sample size) may produce the pathological phenomenon. If the value of \eqref{diff_risk_full_sub} is negative, we must be cautious about the use of the submodel.
%\end{itemize}
\ \ In this closing section, some complementary issues related to the pair of submodel and divergence are discussed. As explained in Introduction,  it is natural to use the conjugate K-L divergence \eqref{conjugate_KL2} when we treat the affine subspace of $\bar{\mathcal{M}}$ with respect to $m$. Let's call the submodel defined by some linear restrictions in $m$, ``$m$-submodel''. $m$-submodel is always a $m$-affince subspace of $\bar{\mathcal{M}}$, since the restriction $\sum_{i=0}^p m_i =1$ is also linear in $m$. 

Now let $\mathcal{M}^*$ denote a $m$-submodel, and let $\hat{m}$ be MLE for the full model, then we can consider two projections of $\hat{m}$ onto the $\mathcal{M}^*$: One is the projection by K-L divergence 
\[
\hat{m}^{+} =\displaystyle{\operatorname*{\mathrm{argmin}}_{m \in \mathcal{M^*}}} D[ \hat{m} \, | \, m];
\]
The other is the projection by the conjugate K-L divergence
\[
\hat{m}^{*} =\displaystyle{\operatorname*{\mathrm{argmin}}_{m \in \mathcal{M^*}}} D_\Psi[ \hat{m} \, | \, m]
\]
In general, $\hat{m}^{+}$ is not equal to $\hat{m}^{*}$, however in the case of the non-overlapping $m$-aggregation submodel,  $\hat{m}^{+}=\hat{m}^{*}$. This can be proved as follows; Suppose  $\mathcal{M}^*$  is defined by the following $I$ restriction equations (including the equation $m_0 + \cdots + m_p=1$),
\[
\sum_{j=0}^p a_{ij} m_{j} = c_i,\qquad i=1,\ldots, I,
\]
Then by the method of Lagrange multiplier, $\hat{m}^+$ is the solution of 
\[
\frac{\partial \quad}{\partial m_i} \Bigl\{ \sum_{j=0}^p \hat{m}_j \log{(\hat{m}_j/m_j)} - \sum_{k=1}^I \lambda_k\Bigl(\sum_{j=0}^p a_{kj} m_{j} -c_k\Bigr)\Bigr\}=0,\qquad i=0,\ldots,p
\]
or equivalently 
\[
\frac{\hat{m}_i}{m_i}=-\sum_{k=1}^I \lambda_k a_{ki}, \qquad i=0,\ldots,p.
\]
For $i \ne j$, if 
\begin{equation}
\label{same_coeff}
a_{ki} = a_{kj}, \quad 1 \leq \forall k \leq I,
\end{equation}
then 
\begin{equation}
\label{same_ratio_m}
\frac{\hat{m}_i}{m_i}=\frac{\hat{m}_j}{m_j}.
\end{equation}
In the case of the non-overlapping $m$-aggregation model, for any $m_i$ and $m_{j}$ in the same restriction equation \eqref{info_sum}, \eqref{same_coeff} holds, hence \eqref{same_ratio_m} holds. Therefore $\hat{m}^+_i$ that appears in the restriction \eqref{info_sum} is given by 
\begin{equation}
\label{concrete_m^+}
\hat{m}^+_{i_j} = \hat{m}_{i_j} \frac{c}{\sum_{j=1}^s \hat{m}_{i_j}},\qquad j=1,\ldots,s.
\end{equation}
Similarly $\hat{m}^*$ is the solution of 
\[
\frac{\partial \quad}{\partial m_i} \Bigl\{ \sum_{j=0}^p m_j \log{(m_j/\hat{m}_j)} - \sum_{k=1}^I \lambda_k\Bigl(\sum_{j=0}^p a_{kj} m_{j} -c_k\Bigr)\Bigr\}=0,\qquad i=0,\ldots,p,
\]
or equivalently
\[
\log{\frac{m_i}{\hat{m}_i}} =\sum_{k=1}^I \lambda_k a_{ki}-1.
\] 
For the non-overlapping m-aggregation submodel, the same argument as $\hat{m}^+$ holds, and $\hat{m}^*_i$ that appears in the restriction \eqref{info_sum} is also given by 
\begin{equation}
\label{concrete_m^*}
\hat{m}^*_{i_j} = \hat{m}_{i_j} \frac{c}{\sum_{j=1}^s \hat{m}_{i_j}},\qquad j=1,\ldots,s.
\end{equation}

For the overlapping submodel, $\hat{m}^{+}$ and $\hat{m}^{*}$  do not necessarily coincide, and are not simply expressed like \eqref{concrete_m^*} or \eqref{concrete_m^+}. Besides, its risk (the asymptotic expansion up to the second-order) needs to be calculated by another method than the two-stage model. 

Another important class of the submodels is the model given by the linear restriction in $\theta$ of \eqref{def_theta} (say ``$e$-submodel''). For example, in the two-way contingency table, if we assume each row (or column) is independently distributed, then this model is an $e$-submodel. In the two by two contingency table, the fixed odds-ratio produces an $e$-submodel. Notice that $e$-submodel is not an affine subspace of $\bar{\mathcal{M}}$, since the restriction $\sum_{i=0}^p m_i=1$ is not the linear restriction in $\theta$. Hence the use of the projection theorem and the related Pythagorean theorem is not straightforward.

\section*{Acknowledgment}
This work was supported by JSPS KAKENHI Grant Number JP20K11706.
%We really  appreciate the vaulable commnets from the reviewers, which improved the quality  and readability of this article.
%
%
%
%
%
\section*{Declarations}
On behalf of all authors, the corresponding author states that there is no conflict of interest.

%
%
%  Appendix
%
%
%
\section*{Appendix}
%
% 
% Lemma 1
%
%
\begin{lemma}
\label{expectations_difference}
Let $X=(X_0,\ldots,X_{p})$ be the random vector whose distribution is defined as the multinomial distribution with \eqref{full_model} and the sample size $n$. The distribution under the condition $X_i \ne 0,  0\leq \forall i \leq p$ is considered. Let the unconditional and conditional expectations of a random variable $Y(X)$ be denoted by $E[Y(X)]$ and $E^*[Y(X)]$, respectively. If 
$$
|Y(X)| \leq  a+b n^c 
$$
holds with some nonnegative numbers $a, b, c$, the difference between the two expectations decreases to zero with exponential speed as $n$ goes to infinity, namely for any $s>0$,
\begin{equation}
\label{diff_cond_uncond}
E^*[Y(X)]=E[Y(X)]+o(n^{-s}).
\end{equation}
In the special case, the following equations hold for any $s>0$;
\begin{align*}
E^*\bigl[D [\widehat{m}(X) : m ]\bigr] &=E \bigl[D [\widehat{m}(X) : m ]\bigr]+o(n^{-s}),\\
E^*[(X_i- n m_i)^k] &=  E[(X_i- n m_i)^k] +o(n^{-s}),\quad  i=1,\ldots,p+1,\ k=1,2,\ldots,
\end{align*}
where  $\widehat{m}(X)$ is the MLE of $m$.
\end{lemma}
\begin{proof}
For $i (=1,\ldots,p+1)$ and $s (>0)$, because the following equivalence relationship holds
\begin{align*}
&n^s (1-m_i)^n \to 0 \\
&\iff s\log n +n \log(1-m_i) \to -\infty \iff n\Bigl(s\frac{\log n}{n}+\log(1-m_i)\Bigr) \to -\infty,
\end{align*}
we have
\begin{equation}
P(X_i=0)=(1-m_i)^n=o(n^{-s}).
\end{equation}
Let $Z_+^{p+1}$ be the set of all $p+1$-dimensional vectors whose elements are nonnegative integers.
\begin{align*}
\mathcal{X}&\triangleq \Bigl\{x=(x_1,\ldots,x_{p+1}) \in Z_+^{p+1} \Big| \sum_{j=1}^{p+1} x_j =n,\ x_i >0,\ 1\leq \forall i \leq p+1\Bigr\}\\
\mathcal{X}^*&\triangleq \Bigl\{x=(x_1,\ldots,x_{p+1}) \in Z_+^{p+1} \Big| \sum_{j=1}^{p+1} x_j =n, x_i=0,\ 1 \leq \exists i \leq p+1\Bigr\}
\end{align*}
Notice that
\[
m^*\triangleq P(X_i=0,\ 1\leq \exists i \leq p+1) \leq \sum_{i=1}^{p+1} P(X_i=0),
\]
which means $m^*=o(n^{-s})$. Because
\begin{align*}
E^*[Y(X)]&=\sum_{x \in \mathcal{X}} Y(x) P(X=x)/(1-m^*) \\
E[Y(X)]&=\sum_{x \in \mathcal{X}} Y(x) P(X=x) +\sum_{x \in \mathcal{X}^*} Y(x) P(X=x),
\end{align*}
we have
\[
E^*[Y(X)]-E[Y(X)]=\sum_{x \in \mathcal{X}} Y(x) P(X=x)\frac{m^*}{1-m^*}-\sum_{x \in \mathcal{X}^*} Y(x) P(X=x).
\]
Choose arbitrary $s(>0)$. Because $\mathcal{X}$ and $\mathcal{X}^*$ are both finite sets, for some nonnegative constants $a', a'', b', b''$, 
\begin{align*}
n^s\Bigl|\sum_{x \in \mathcal{X}} Y(x) P(X=x)\frac{m^*}{1-m^*} \Bigr|
&\leq n^s \sum_{x \in \mathcal{X}} (a+b n^c) \frac{m^*}{1-m^*}\\
&\leq a' \frac{n^s m^*}{1-m^*} + b' \frac{n^{s+c} m^*}{1-m^*}\\
n^s\Bigl|\sum_{x \in \mathcal{X}^*} Y(x) P(X=x) \Bigr|\leq n^s\sum_{x \in \mathcal{X}^*} (a+b n^c) m^*
&\leq a''n^s m^*+b'' (n^{s+c} m^*).
\end{align*}
Because $m^*=o(n^{-t})$ for any $t(>0)$, 
\[
\frac{n^sm^*}{1-m^*}, \frac{n^{s+c} m^*}{1-m^*}, n^s m^*, n^{s+c} m^* \to 0
\]
as $n\to \infty$, which shows \eqref{diff_cond_uncond}. The rest is obvious from the fact 
$D[\widehat{m}(X) : m ]$ is bounded (notice that $x\log x \to 0$ as $x \to 0$) and
\[
|(X_i - n m_i)^k| \leq \max (m_i^k, (1-m_i)^k) n^k,\quad k=1,2,\ldots.
\]
\end{proof}

\end{document}